\documentclass{amsart}
\usepackage{amssymb}
\usepackage{xcolor}
\usepackage{mathrsfs}
\usepackage{amsfonts}
\usepackage{amsmath}
\usepackage{tikz}
\usepackage[compress]{cite}
\numberwithin{equation}{section}
\allowdisplaybreaks[4]
\newtheorem{Theorem}{Theorem}[section]
\newtheorem{Proposition}[Theorem]{Proposition}

\newtheorem{Lemma}[Theorem]{Lemma}
\newtheorem{Example}[Theorem]{Example}
\newtheorem{Remark}[Theorem]{Remark}
\newtheorem{Definition}[Theorem]{Definition}
\ifcsname proof\endcsname
\else

\fi
\def\rank{\mathrm{rank}\,}

\def\Lspan{\mathrm{span}\,}
\def\bbZ{\mathbb{Z}}
\def\bbR{\mathbb{R}}

\def\calP{\mathcal P}

\def\supp{\mathrm{supp}\ }

\def\Lspan{\mathrm{span}\ }
\def\Vm{\mathcal V_m}
\def\Wm{\mathcal W_m}
\def\tVm{\tilde{\mathcal V}_m}
\def\Vmk{\mathcal V_m|_{[t_0,t_k]}}
\def\tVmk{\tilde{\mathcal V}_m|_{[t_0,t_k]}}

\begin{document}

\title[Phaseless Sampling and Linear Reconstruction]{Phaseless Sampling and Linear Reconstruction of Functions in Spline Spaces}

\thanks{This work was partially supported by the
National Natural Science Foundation of China (11371200, 11525104 and  11531013).}

\author{Wenchang Sun}

\address{School of Mathematical Sciences and LPMC,  Nankai University,       Tianjin~300071, China}

\email{sunwch@nankai.edu.cn}

\begin{abstract}
We study phaseless sampling in spline spaces generated by B-splines with arbitrary knots.
For real spline spaces, we give a necessary and sufficient condition for a sequence of sampling points to
admit a local phase retrieval of any nonseparable function.
We also study phaseless sampling in complex spline spaces and illustrate that phase retrieval is impossible in this case.
Nevertheless, we show that  phaseless sampling is possible.
For any function $f$ in a complex spline space, no mater it is separable or not,
we show that $|f(x)|^2$ is uniquely determined and can be recovered linearly from its  sampled values at a well chosen sequence of sampling points.
We give necessary and sufficient conditions for such sequences.
\end{abstract}



\maketitle

\section{Introduction and Main Results}

In this paper, we study phaseless sampling in real and complex spline spaces.
This problem arises in the study of the classical sampling theory in shift invariant spaces
and some recent studies on  phase retrieval for general frames.

Phase retrieval arises in the study of recovering  signals from intensity measurements\cite{Shech2015}.
In the setting of frame theory, this problem was introduced by Balan, Casazza  and Edidin \cite{BaCE2006}.
It was shown that with  certain frames, we can recover a signal up to a phase from its phaseless frame coefficients.
We refer to  \cite{Bah2014,BaCE2006,Band2014a,Band2014b,Ca2015,CCPW,Ca2017,CLNZ2016,
Ed2017,Gross2017,Huang2016,Iwen2017,Iwen2016,Qian2016,Qiu2016,Tillmann2016}
for various aspects on this topic.

Recently, many works for phase retrieval in infinite-dimensional spaces have been done
\cite{Al2016,Al2016b,Cahill2016,Ma2015,Po2014a,Po2014b,Sh2016,Th2011,yang2013}.
When the intensity measurements come  from sampled values of a function, this problem is studied with the term phaseless sampling.
Chen, Cheng, Jiang, Sun and Wang  \cite{Chen2016,Cheng2017} studied phaseless sampling and reconstruction of
real-valued signals.
They show that phase retrieval is possible for nonseparable functions  in general shift-invariant spaces
whenever the sampling points are sufficiently dense.
As a consequence, for the shift invariant space
\[
  V_m = \left\{\sum c_n \varphi_m(x-n):\,     c_n    \in  \mathbb R,  n\in\mathbb Z\right\}
\]
generated by  the $m$-degree B-spline
\[
   \varphi_m = \chi_{[0,  1]}^{ } * \ldots * \chi_{[0, 1]}^{ }  \quad\quad  ( \mbox{ $m+1$ terms}),\quad  m\ge 1,
\]
it was shown in \cite{Chen2016} that any nonseparable function $f\in V_m$ is uniquely determined up to a sign
from its unsigned sampled values at a sequence $E\subset \bbR$ which satisfies $\#(E\cap (n,n+1)) \ge 2m+1$.

In \cite{Sun2017}, we gave a characterization of phaseless sampling sequences for $V_m$.
Specifically, we gave a necessary and sufficient condition on the sampling sequence $\{x_n:\,n\in\bbZ\}$
such that any nonseparable function $f\in V_m$ is uniquely determined up to a sign
from its unsigned sampled values $|f(x_n)|$, $n\in\bbZ$.

Since apparatuses can process only finitely many data, local sampling and reconstruction of signals
are practically useful. In this case, we need to recover a signal locally from finitely many samples.
For  local phaseless sampling of functions in real spline spaces, we also get
a necessary and sufficient condition for a sequence of sampling point $\{x_n:\,1\le n\le N\}$ to admit
a local phaseless recovery, i.e., to recover $f$ up to  a sign on an interval  from its unsigned samples $|f(x_n)|$, $1\le n\le N$.

 In this paper, we generalize this result to  more general settings.
In fact, we give a complete characterization of local phaseless sampling sequences for functions in real spline spaces generated by
B-splines with arbitrary knots.

Stability is required in practice.
For the stable recovery of signals from phaseless measurements with general frames,
Bodmann and Hammen \cite{Bod2016} provided a recovery algorithm with explicit error bounds
whenever the frame contains at least $6d-3$ elements, where $d$ stands for the dimension of signals.
In this paper, we show that linear and therefore stable recovery of functions in spline spaces
from phaseless sampled values is possible with much less samples.

Specifically, we study the stable recovery of functions in complex spline spaces from phaseless sampled values.
Note that we have to consider more cases for complex functions.
For example, both $f$ and its conjugate $\overline{f}$ have the same phaseless samples. Therefore we can not differ $\overline f$
from $f$ with phaseless samples.
The following is an explicit example.

\begin{Example}\label{ex:x1}\upshape
Take $m=1$ and $N\ge 1$. Set
\[
   f(x) = \sum_{n=0}^N ((-1)^n+\sqrt{-1}) \varphi_1(x-n).
\]
Then we have
\[
  |f(x)|^2 = \sum_{n=0}^N 2 |\varphi_1(x-n)|^2.
\]
Let $\{a_n:\, 0\le n\le N\}$  be a sequence of $\pm 1$, say, $\{1, -1,-1, -1,\ldots\}$. Define
$b_1 =a_1$ and $b_n = - a_na_{n-1}/b_{n-1}$ for $n\ge 2$. Then we have
$|a_n|=|b_n|=1$ and $a_{n}a_{n+1}+b_nb_{n+1}=0$ for $1\le n\le N$.
Let
\[
  g(x) = \sum_{n=0}^N (a_n + b_n\sqrt{-1})  \varphi_1(x-n).
\]
Then we have
\[
  |f(x)|^2 = |g(x)|^2, \qquad x\in\bbR.
\]
Obviously, there is no constant $c$ with $|c|=1$ such that $f = c\cdot g$ or $f = \overline{c\cdot g}$.
\end{Example}

Let $W_m$ be defined similarly as $V_m$ except that the coefficients $c_n$ are complex numbers.
The above  example shows that  phase retrieval is impossible in $W_m$,
that is, we can not determine $f$ or $\overline{f}$ up to a phase from phaseless sampled values for $f\in W_m$.
Nevertheless, there is a phaseless sampling theorem for $W_m$.
Moreover, we show that it is possible to reconstruct $|f(x)|^2$ linearly from its sampled values.
Before stating our result, we introduce some definitions.

Let $H$ be a set of functions.   Define $|H|=\{|f|:\, f\in H\}$.
Let  $|H|_{[a,b]}$ be the restriction of $|H|$ on $[a,b]$.

\begin{Definition}
Let $H$ be a set of functions.
We call $E=\{x_n:\,1\le n\le N\} \subset [a, b]$  a   \emph{phaseless sampling sequence} for $H|_{[a,b]}$ if for any nonseparable $f\in H|_{[a,b]}$,
$f$ is uniquely determined up to a phase by its phaseless sampled  values over $E$.
Recall that a function $f\in H$ is said to be separable if $f=f_1+f_2$ for some $f_1, f_2\in H\setminus\{0\}$
with $f_1(x)f_2(x) = 0$.

And we call $E$ a \emph{sampling sequence} for $H|_{[a,b]}$ if any $f\in H|_{[a,b]}$ is uniquely determined by its
 sampled  values over $E$.

Moreover, if there exists a sequence of functions $\{S_n:\,1\le n\le N\}$ such that
\[
  f(x) = \sum_{n=1}^N f(x_n) S_n(x),\qquad  \forall x\in [a,b],\,\, f\in H,
\]
then we call $E$ a \emph{linear sampling sequence}  for $H|_{[a,b]}$.
\end{Definition}

In this paper, we study local sampling in $|W_m|^2$.
Specifically, we search for conditions on the sequence of sampling points $E$ such that any function in  $|W_m|^2_{[n_1,n_2]}$ is uniquely determined
by its sampled values on $E$, where $n_1<n_2$ are integers.
Although phase retrieval is  impossible in $W_m$, we show that sampling sequences for  $|W_m|^2_{[n_1,n_2]}$  do exist.
Moreover, it is possible to reconstruct $|f(x)|^2$ linearly  from its sampled values   with well chosen sampling points.
That is, there exist linear sampling sequences for  $|W_m|^2_{[n_1,n_2]}$.
In fact, we give necessary and sufficient conditions for such sequences of points.

\begin{Theorem}\label{thm:linear}
Let $n_1<n_2$ and $N$ be integers. Let $E=\{x_n:\, 1\le n\le N\} \subset [n_1, n_2]$  be a sequence  consisting of distinct points.
Then
$E$ is a   linear sampling sequence for $|W_m|^2_{[n_1,n_2]}$  if and only if it  satisfies the followings,
\begin{align}
\# E &\ge (m+1)(n_2-n_1)+m,          \label{eqn:linear:1}               \\
\# (E\cap [n_1, i)) &\ge (m+1)(i-n_1),         \quad   n_1<  i \le  n_2,     \label{eqn:linear:2}  \\
\# (E\cap (i, n_2]) &\ge (m+1)(n_2-i),       \quad   n_1\le  i < n_2,  \label{eqn:linear:3} \\
\# (E\cap(i,j)) &\ge  (m+1)(j-i)-m,             \,\,        n_1\le i<j\le n_2,           \label{eqn:linear:4}
\end{align}
where $\# E$ denotes the cardinality of  $E$.

When the conditions are satisfied, there exist functions $S_n$, $1\le n\le N$, such that
\[
  |f(x)|^2  = \sum_{n=1}^N |f(x_n)|^2 S_n(x),\qquad \forall  x\in [N_1, N_2], \  f\in W_m.
\]
\end{Theorem}

We see from the above theorem that  $(m+1)(N_2-N_1)+m$ is the minimal cardinality for a sequence to be a linear sampling sequence.
Note that the dimension of $W_m|_{[N_1,N_2]}$ is $d^{}_{N_1,N_2} = N_2-N_1+m$. Hence
\[
  (m+1)(N_2-N_1)+m = (m+1) d^{}_{N_1,N_2} - m^2,
\]
which is less than $4d^{}_{N_1,N_2} - 4$ whenever $m\le 3$. Therefore, it is impossible for phase retrieval in this case.

Recall that to get a stable recovery of signals from phaseless measurements with general frames,
we need  at least $6d-3$ frame coefficients \cite{Bod2016}.
The above theorem shows that we can get a linear recovery with $(m+1) d^{}_{N_1,N_2} - m^2$ sampled values.
Since the sample number is less than $6d^{}_{N_1,N_2}-3$ whenever $m\le 6$,
this shows that phaseless sampling is quite different from phase retrieval.

On the other hand, observe that for $m=1$, conditions (\ref{eqn:linear:1}) - (\ref{eqn:linear:4}) coincide with \cite[Theorem 1.2]{Sun2017}.
Therefore, $E$ is a linear sampling sequence for $|W_1|^2_{[N_1,N_2]}$ if and only if it is a phaseless sampling sequence for $V_1|_{[N_1,N_2]}$.

The paper is organized as follows.
In Section 2, we generalize a result on local phaseless sampling in \cite{Sun2017}.
For the real function space generated by B-splines with arbitrary knots, we give a necessary and sufficient condition
for  a sequence of points to be a local phaseless sampling sequence.
In Section 3, we give a linear phaseless sampling theorem in complex spline spaces, of which Theorem~\ref{thm:linear} is a consequence.
We also give an explicit reconstruction formula for  linearly recovering   functions from phaseless sampled values.

\section{Phaseless Sampling in Real Spline Spaces with Arbitrary Knots}

In this section, we give a characterization of phaseless sampling sequences
for  the real function space generated by B-splines with arbitrary knots.
We begin with some preliminary results.

Suppose that  $\{t_i:\, i\in\bbZ\}$ is a sequence of real numbers such that
\[
  t_i<t_{i+1},\qquad i\in\bbZ.
\]
Let $m$ be a positive integer and $\{m_i:\, i\in\bbZ\}$ be a sequence of positive integers such that
$m_i\le m$.

For each $i\in\bbZ$ and $1\le l\le m_i$,
let $Q_{i,l}$ be the $m$-degree B-spline with knot sequence $(t_i,\ldots, t_i, t_{i+1}, \ldots, t_{i+n})$,
where $t_i$ appears  $l$ times,
$t_j$ appears $m_j$ times for $i+1\le j\le i+n-1$,
and $t_{i+n}$ appears $m+2 - (l+m_{i+1}+\ldots +m_{i+n-1})$ times.

For simplicity, we denote by  $R_i$   the B-spline with knot sequence $(t_{i-n}$, $\ldots$, $t_{i-1},t_i)$,
where
$t_j$ appears $m_j$ times for $i-n+1\le j\le i-1$,
and $t_{i-n}$ appears $m+2 - (1+m_{i-1}+\ldots +m_{i-n+1})$ times.

Define
\[
  \Vm =\Big  \{\sum_{i\in\bbZ}\sum_{1\le l\le m_i} c_{i,l}Q_{i,l}:\, c_{i,l}\in\bbR \Big\}.
\]
Since for $x\in\bbR$, there are at most $m+1$ index pairs $(i,l)$ satisfying $Q_{i,l}(x)\ne 0$,
$\Vm$ is well defined.

Fix some $k\ge 1$. Denote
\[
   I_{0,k} = \{(i,l):\, (t_0,t_k)\cap \supp Q_{i,l}  \ne \emptyset \}.
\]
It is easy to see that $I_{0,k}$ contains exactly
\[
  N_{0,k}:=m+1+\sum_{l=1}^{k-1} m_l
\]
 elements
and
  there exist integers $i_0$, $i_k$ and $n_i$ such that
\[
   I_{0,k} = \{(i,l):\, i_0\le i\le i_k,   1\le l\le n_i\}.
\]
Convention: elements of $I_{0,k}$ are listed in the following order
$I_{0,k} = \{(i_n,l_n):\, 1\le n\le N_{0,k}\}$,
where $i_n < i_{n+1}$ or $i_n=i_{n+1}$ and $l_n >l_{n+1}$.

In general, if two pairs of indices $(i,l)$ and $(i',l')$ satisfy
$i<i'$ or $i=i'$ and $l>l'$, then we write
\[
  (i,l) < (i', l').
\]

The following lemma gives a basis and also the dimension of $\Vmk$.

\begin{Lemma}\label{Lem:spline:basis}
Let $k\ge 1$ be an integer.
Then $\{Q_{i,l}|_{[t_0,t_k]}:\, (i,l)\in I_{0,k}\}$ is a basis for
$\Vm|_{[t_0, t_k]}$.
\end{Lemma}

\begin{proof}
By \cite[Theorem 4.18]{Sch},
$\{Q_{i,l}:\, (i,l)\in I_{0,k}\}$ is linearly independent on $[t_0, t_k]$.
Since $\Vm|_{[t_0, t_k]}$ is the linear span of these B-splines, the conclusion follows.
\end{proof}

As shown in \cite{Chen2016}, it is impossible to do phase retrieval from phaseless sampled values for separable functions.
In the following, we give some characterization of separable functions in $\Vm$.
\begin{Lemma}\label{Lem:sep}
Let $k\ge 1$ be an integer
and  $f\in \Vmk$. Then the following three assertions are equivalent.

\begin{enumerate}
\item $f$ is separable.

\item $k\ge 2$, $f = \sum_{(i,l)\in  I_{0,k}} c_{n,l}Q_{i,l}$ for some sequence $\{c_{i,l}:\, (i,l)\in  I_{0,k}\} \subset \bbR$,
and there exists some integer $n_0\in (0,k)$ such that
$f|_{[t_0, t_{n_0}]} \ne 0$,
$f|_{[t_{n_0}, t_k]} \ne 0$,
and
$c_{i,l}=0$ whenever $ (t_{n_0-1}, t_{n_0+1}) \subset \supp Q_{i,l}$.

\item There exists some integer $n_0\in (0,k)$ and $f_1, f_2\in \Vmk\setminus\{0\}$ such that
$f=f_1+f_2$, $\supp f_1 \subset (t_0, t_{n_0})$
and $\supp f_2 \subset (t_{n_0}, t_k)$.

\end{enumerate}
\end{Lemma}

\begin{proof}
Since
$\{Q_{i,l}|_{[t_0,t_k]}:\, (i,l)\in I_{0,k}\}$ is a basis for
$\Vm|_{[t_0, t_k]}$,  for any $f\in \Vmk$,
there exist some sequence $\{c_{i,l}:\, (i,l)\in  I_{0,k}\} \subset \bbR$ such that
$f = \sum_{(i,l)\in  I_{0,k}} c_{n,l}Q_{i,l}$.

(i) $\Rightarrow$ (ii).\,\,
Assume that $f$ is separable, that is, there exist $f_1, f_2\in \Vmk\setminus\{0\}$ such that
$f = f_1 + f_2$ and $f_1(x)f_2(x) = 0$ for all $x\in [t_0, t_k]$.

Since $f_1, f_2\ne 0$, there exist integers $n_1, n_2 \in [0, k-1]$ such that
\[
  f_1|_{[t_{n_2},t_{n_2+1}]} \ne 0 \quad \mathrm{and} \quad f_2|_{[t_{n_1},t_{n_1+1}]} \ne 0.
\]
On the other hand, observe that both $f_1$ and $f_2$ are polynomials on $[t_n,t_{n+1}]$.
We see from $f_1(x) f_2(x) \equiv 0$ that
$f_1$ or $f_2$ must be identical to zero on $[t_n,t_{n+1}]$ for every $n\in [0, k-1]$.
Hence
\[
  f_1|_{[t_{n_1},t_{n_1+1}]} =0 \quad \mathrm{and} \quad f_2|_{[t_{n_2},t_{n_2+1}]} = 0.
\]
Consequently, $n_1\ne n_2$ and $k\ge 2$. Without loss of generality, assume that $n_1<n_2$.

Let
\begin{equation}\label{eqn:sep:e1}
  n_0 = \max \{n:\, f_1|_{[t_{n_1},t_n]}=0\}.
\end{equation}
Then we have $n_1<n_0 \le n_2$ and $f_1|_{[t_{n_0},t_{n_0+1}]}\ne 0$. Since $ f_1f_2=0$, we have
$f_2|_{[t_{n_0},t_{n_0+1}]}=0$.

Suppose that
$f_p = \sum_{(i,l)\in I_{0,k}} c_{i,l}^{(p)} Q_{i,l}$, $p=1,2$.
Since $f_1|_{[t_{n_0-1, t_{n_0}}]} = 0$ and
$f_2|_{[t_{n_0,n_0+1}]} = 0$, we see from Lemma~\ref{Lem:spline:basis} that
whenever $(t_{n_0-1}, t_{n_0+1}) \subset \supp Q_{i,l}$,
\[
  c_{i,l}^{(p)} = 0, \qquad p=1,2.
\]
Hence
\[
  c_{i,l} = c_{i,l}^{(1)} + c_{i,l}^{(2)}= 0.
\]

On the other hand, since  $f_1|_{[t_{n_0},t_{n_0+1}]}\ne 0$, we have $f|_{[t_{n_0},t_k]}\ne 0$.
Similarly we get
$f|_{[t_0, t_{n_0}]}\ne 0$ since
$f_2|_{[t_{n_1},t_{n_1+1}]} \ne 0$ and $n_1<n_0$.
Now we get the conclusion as desired.

(ii) $\Rightarrow$ (iii).
Let
\[
  f_1 = \sum_{\substack{(i,l)\in I_{0,k} \\ (t_{n_0}, t_k) \cap \supp Q_{i,l} =\emptyset }} c_{i,l} Q_{i,l},
  \qquad
  f_2 = \sum_{\substack{(i,l)\in I_{0,k} \\ (t_0, t_{n_0}) \cap \supp Q_{i,l} =\emptyset }} c_{i,l} Q_{i,l}.
\]
Then we have
$\supp f_1 \subset (t_0, t_{n_0})$
and $\supp f_2 \subset (t_{n_0}, t_k)$.
For any $(i,l)\in I_{0,k}$, there are three cases:
$(t_0, t_{n_0}) \cap \supp Q_{i,l} =\emptyset$,
$(t_{n_0-1}, t_{n_0+1})$ $\subset \supp Q_{i,l}$, or
$(t_{n_0}, t_k) \cap \supp Q_{i,l} =\emptyset$.
Hence $f=f_1+f_2$.
Moreover, $f_1|_{[t_0, t_{n_0}]} = f|_{[t_0, t_{n_0}]} \ne0$
and
$f_2|_{[t_{n_0},t_k]} = f|_{[t_{n_0},t_k]} \ne 0$.

(iii) $\Rightarrow$ (i) is obvious. This completes the proof.
\end{proof}

\begin{Remark}\upshape
Fix some integer $n\in (0, k)$.
The B-spline $Q_{i,l}$ which satisfies $(t_{n-1},t_{n+1})\subset \supp Q_{i,l}$
is exactly the B-spline whose knot sequence is of the form $(\ldots,t_{n-1},t_n,\ldots,t_n,t_{n+1},\ldots)$.
It is easy to see that the cardinality of the set consisting of such B-splines
is $m-m_n+1\ge 1$.
\end{Remark}

Next we introduce a result on local sampling in spline spaces,
which gives a characterization of  linear sampling sequences for $\Vmk$.

\begin{Proposition}[{\cite[Theorem 1.1]{Sun09}}]\label{prop:local sampling}
A sequence $E\subset [t_0, t_k]$ consisting of distinct points
is a linear sampling sequence for $\Vmk$ if and only if it satisfies the followings,
\begin{align}
\#  E                     &\ge m+1+\sum_{l=1}^{k-1} m_l,                                    \\
\# (E\cap [t_0, t_i))     &\ge  \sum_{l=1}^i m_l,       \quad   1\le i\le k,     \\
\# (E\cap (t_i, t_k]) &\ge  \sum_{l=i}^{k-1} m_l,       \quad   0\le i\le k-1,    \\
\# (E\cap(t_i, t_j))      &\ge  \sum_{l=i}^j m_l -m-1,     \quad   0\le i<j\le k,
\end{align}
\end{Proposition}

The following result shows that if a sequence contains sufficiently many points,
it must be a linear sampling sequence for $\Vm$ on some interval.

\begin{Lemma}\label{Lem:samp}
Let $k\ge 1$ be an integer and  $E\subset [t_0, t_k]$.
If $\# E\ge m+1+\sum_{l=1}^{k-1} m_l$, then there exist some integers $n_1, n_2 \in [0, k]$
such that $n_1<n_2$ and $E\cap[t_{n_1}, t_{n_2}]$ is a linear sampling sequence for $\Vm|_{[t_{n_1}, t_{n_2}]}$.
\end{Lemma}

\begin{proof}
We prove the conclusion with induction over $k$.
We see from Proposition~\ref{prop:local sampling} that it is the case if $k=1$.

Now we assume that  for some $n\ge 1$, the conclusion is true  for all $1\le k\le n$.
Let us consider the case of $k=n+1$.

Assume that $E\subset [t_0, t_k]$,
$\#E  \ge m+1+\sum_{l=1}^{k-1} m_l$,
and $E$  is not a linear sampling sequence for $\Vmk$.
By Proposition~\ref{prop:local sampling}, there are three cases.

(i).\,\, There is some integer $i\in [1, k]$ such that  $\#(E\cap[t_0, t_i))<\sum_{l=1}^i m_l$.

In this case, we have $i<k$ and  $\#(E\cap [t_i, t_k]) \ge m+1+\sum_{l=i+1}^{k-1} m_l$.
Replacing $(t_0,t_k)$ with $(t_i, t_k)$, we see from the inductive assumption  that there exist
some integers $n_1,n_2 \in [i, k]$  such that $n_1<n_2$ and
$E\cap[t_{n_1}, t_{n_2}]$ is a linear sampling sequence for $\Vm|_{[t_{n_1}, t_{n_2}]}$.

(ii).\,\, There is some integer $i \in [0, k-1]$ such that  $\#(E\cap(t_i, t_k] )< \sum_{l=i}^{k-1} m_l$.

Similarly to the previous case we can prove the conclusion.

(iii).\,\, $\#(E\cap[t_0, t_i))\ge \sum_{l=1}^i m_l$ and
$\#(E\cap(t_{i-1}, t_k] ) \ge  \sum_{l=i-1}^{k-1} m_l$  for any $1\le i\le k$.

If for any integers $i<j$ with $[i,j] \subset [0,k]$,
we have $\# (E\cap(t_i, t_j))      \ge  \sum_{l=i}^j m_l -m-1$, then we see from Proposition~\ref{prop:local sampling}
that $E$ is a linear sampling sequence for $\Vmk$.
If  $\# (E\cap(t_i, t_j))      <  \sum_{l=i}^j m_l -m-1$ for some $i<j$,
then we have $\sum_{l=i}^j m_l\ge m+2$ and $(i,j)\ne (0,k)$.
Without loss of generality, we assume that $i>0$. Then
\begin{eqnarray*}
  \#(E\cap [t_0, t_i]) &=& \#(E\cap [t_0, t_j)) - \#(E\cap (t_i, t_j)) \\
  &\ge& \Big(\sum_{l=1}^j m_l \Big)-\Big(\sum_{l=i}^j m_l -m-1\Big)= \sum_{l=1}^{i-1}m_l +m+1.
\end{eqnarray*}
Observe that $i\le k-1= n$.
By the inductive assumption, there exist some integers $n_1<n_2$ in $[0, i]$ such that $E\cap[t_{n_1},t_{n_2}]$ is a linear sampling sequence
for $\Vm|_{[t_{n_1}, t_{n_2}]}$.

By induction, the conclusion is true for any $k\ge 1$.
\end{proof}

To prove the main result, we also need the following version of the Sch\"onberg-Whitney Theorem.

\begin{Proposition}[{\cite[Theorem 4.65]{Sch}}]\label{prop:Sch}
Let $k\ge 1$ be an integer and $x_1<\ldots<x_{N_{0,k}}$ be $N_{0,k}$ real numbers.
Then the $N_{0,k}\times N_{0,k}$ matrix
\[
  [ Q_{i,l}(x_n)]_{1\le n\le N_{0,k}, (i,l)\in I_{0,k}}
\]
 is invertible if and only if
none of its diagonal entries is equal to zero.
\end{Proposition}

We are now ready to give a complete characterization of local
phaseless sampling sequences for $\Vm$.

\begin{Theorem} \label{Thm:phaseless:local}
A sequence $E\subset [t_0, t_k]$ consisting of distinct points
is a phaseless sampling sequence for $|\Vm|^2_{[t_0,t_k]}$ if and only if it satisfies the followings,
\begin{align}
\#  E                     &\ge 2m+1+\sum_{l=1}^{k-1} 2m_l,                                                   \label{eqn:PR:1} \\
\# (E\cap [t_0, t_n))     &\ge  m+m_n+\sum_{l=1}^{n-1} 2m_l,       \quad   1\le n\le k,    \label{eqn:PR:2} \\
\# (E\cap (t_n, t_k]) &\ge  m+m_n+\sum_{l=n+1}^{k-1} 2m_l,       \quad   0\le n\le k-1,      \label{eqn:PR:3} \\
\# (E\cap(t_{n_1}, t_{n_2}))      &\ge m_{n_1} + m_{n_2}-1+ \sum_{l=n_1+1}^{n_2-1} 2 m_l,     \quad   0\le n_1<n_2 \le k,  \label{eqn:PR:4}
\end{align}
\end{Theorem}

\begin{proof}[Proof of the necessity of Theorem~\ref{Thm:phaseless:local}]
Denote $E=\{x_n:\, 1\le n\le N\}$.
First, we prove that
\begin{eqnarray}
\# (E\cap [t_0, t_1)) &\ge&  m + m_1,  \quad       \label{eqn:PR:2a}  \\
\# (E\cap (t_{k-1}, t_k]) &\ge& m +m_{k-1}.  \quad       \label{eqn:PR:3a}
\end{eqnarray}
Assume on the contrary that
(\ref{eqn:PR:2a}) is false.  Then $\#(E\cap[t_0, t_1)) \le m+m_1-1$.
To get a contradiction, we only need to consider the case of $\#(E\cap[t_0, t_1)) = m+m_1-1$.
Split $E$ into two subsequences $E_1$ and $E_2$ such that $E_1 \subset [t_0, t_1)$,
$\#E_1 = m$ and $\#(E_2\cap [t_0, t_1))= m_1-1$.
Define the $m\times N_{0,k}$ matrix $A_1$
and the $\#E_2\times N_{0,k}$ matrix $A_2$ respectively by
\begin{eqnarray}
  A_1 &=& [Q_{i,l}(x_n)]_{x_n\in E_1, (i,l)\in I_{0,k}}, \label{eqn:A1a}\\
  A_2 &=& [Q_{i,l}(x_n)]_{x_n\in E_2, (i,l)\in I_{0,k}}. \label{eqn:A2a}
\end{eqnarray}
Note that there are only $m+1$ B-splines whose supports contain the interval $(t_0, t_1)$.
Hence the last $N_{0,k}-m-1$ columns of $A_1$ are zeros.
Since $\rank(A_1) \le m$, there is some $c = (c_{i,l})^t \in \bbR^{N_{0,k}}$  such that
\[
  A_1 c = 0,
\]
all the last $N_{0,k}-m-1$ entries of $c$ are $1$,
and not all of the first $m+1$ entries are zeros.
Moreover, we see from Proposition~\ref{prop:Sch} that
any $m$ vectors in the first $m+1$ column vectors are linearly independent.
Hence none entry of $c$ is zero.

On the other hand, note that for the first $m_1$ B-splines $Q_{i,l}$ with $(i,l)\in I_{0,k}$,
we have $Q_{i,l}(x)=0$ for $x\ge t_1$.
Hence the submatrix of $A_2$ consisting of the first $m_1$ columns
has at most $\#E_2\cap [t_0,t_1)$ non-zero rows.
Therefore, its rank is no greater than $m_1-1$.
  Consequently, there is some $c' = (c'_{i,l})^t \in \bbR^{N_{0,k}}$,
 not all of whose first $m_1$ entries are zeros,  such that
\[
  A_2 c' = 0.
\]
Let
\begin{equation}\label{eqn:fg}
\left\{
\begin{aligned}
f_1 &= \sum_{(i,l)\in I_{0,k}} \frac{1}{2}(c_{i,l} + c'_{i,l}) Q_{i,l},  \\
f_2 &= \sum_{(i,l)\in I_{0,k}} \frac{1}{2}(c_{i,l} - c'_{i,l}) Q_{i,l}.
\end{aligned}\right.
\end{equation}
By multiplying a factor we can suppose that
\begin{equation}\label{eqn:cc1}
\min\{|c'_{i,l}|:\,    |c'_{i,l}|>0\}  > \max\{|c_{i,l}|:\,    |c_{i,l}|>0\}.
\end{equation}
As a result, none of entries of $c\pm c'$ is zero. By Lemma~\ref{Lem:sep},
both $f_1$ and $f_2$ are nonseparable.  Moreover,  $f_1\pm f_2 \ne 0$ and
\begin{eqnarray*}
   f_1(x_n) &=& -f_2(x_n),\qquad x_n\in E_1, \\
   f_1(x_n) &=&  f_2(x_n),\qquad x_n\in E_2.
\end{eqnarray*}
Hence we can not recover $f$ from unsigned samples, which contradicts with the hypothesis.
This proves (\ref{eqn:PR:2a}).
Similarly we can prove (\ref{eqn:PR:3a}).

{\color{blue} Next we prove (\ref{eqn:PR:4}).} First, we show that for any integer $p\in [0,k-1]$,
\begin{equation}\label{eqn:n}
\# (E\cap (t_p,t_{p+1})) \ge m_p+m_{p+1}-1.
\end{equation}
We see from (\ref{eqn:PR:2a}) and (\ref{eqn:PR:3a}) that (\ref{eqn:n}) is true for $p=0$ or $k-1$.
Next we assume that $\# (E\cap (t_p, t_{p+1})) = m_p+m_{p+1}-2$ for some $p\in [1,k-2]$.
Split $E$ into two subsequences $E_1$ and $E_2$ such that
$E_1\subset [t_0, t_{p+1})$,
$E_2\subset (t_p, t_k]$,
$\#(E_1\cap (t_p, t_{p+1})) = m_p-1$
and $\#(E_2\cap (t_p, t_{p+1})) = m_{p+1}-1$.

Define the matrices $A_1$ and $A_2$ by (\ref{eqn:A1a}) and (\ref{eqn:A2a}), respectively.
Then the right $\sum_{l=p+1}^{k-1} m_l$ columns of $A_1$
and the left $\sum_{l=1}^{p} m_l$ columns of $A_2$ are zeros.
Moreover, the submatrix of $A_1$ consisting of the
$(1+\sum_{l=p+1}^{k-1} m_l)$-th, $\ldots$,
$(m_p+\sum_{l=p+1}^{k-1} m_l)$-th columns on the right
is of the form $\begin{pmatrix}
0 \\ A_{11}
\end{pmatrix}$, where $A_{11}$ is an $(m_p-1)\times m_p$ matrix
and any $(m_p-1)\times(m_p-1)$ submatrix of $A_{11}$ is nonsingular, thanks to Proposition~\ref{prop:Sch}.
Consequently, there is some $c \in \bbR^{N_{0,k}}$, for which the first
$m+1+\sum_{l=1}^{p-1} m_l$ entries are zeros and none of the rest is zero, such that
\[
  A_1 c =0.
\]
Note  that such a vector $c$ also exists if $m_p=1$.

Similarly, we can find some
$c' \in \bbR^{N_{0,k}}$, for which the last
$m+1+\sum_{l=p+2}^{k-1} m_l$ entries are zeros and none of the rest is zero, such that
\[
  A_2 c' =0.
\]
Again, by multiplying a constant, we can suppose that $c_{i,l} \pm c'_{i,l}\ne 0$ if one of
$c_{i,l}$ and $c'_{i,l}$ is not zero.
Let $f_1$ and $f_2$ be defined by (\ref{eqn:fg}).
For $f_1$, we see from the construction that
for any integer $q\in (0,k)$, there exists some $Q_{i,l}$ such that $(t_{q-1},t_{q+1})\subset \supp Q_{i,l}$
and the coefficient of $Q_{i,l}$ is not zero.
By Lemma~\ref{Lem:sep}, $f_1$ is nonseparable.
Similarly we can prove that $f_2$ is also nonseparable.
With similar arguments as in the previous case we get a contradiction.
Moreover, the above arguments also show that $\#(E\cap(t_p, t_{p+1})) < m_p+m_{p+1}-2$ is impossible.
Hence (\ref{eqn:n}) is true.

Now we assume that (\ref{eqn:PR:4}) is false. Then there exist integers  $n_1, n_2\in [0,k]$
such that $n_2-n_1\ge 2$ and $\# (E\cap (t_{n_1}, t_{n_2})) \le m_{n_1}+m_{n_2}-2+\sum_{l=n_1+1}^{n_2-1}2m_l$.
Let
\begin{align}
i_2 &= \min\Big\{i\in [0,k]:\,
  \# (E\cap (t_{i_1}, t_i)) \le m_{i_1}+m_i-2+\sum_{l=i_1+1}^{i-1}2m_l  \label{eqn:n2} \\
  &\qquad\qquad\qquad\qquad \mbox{ \ for some } i_1<i  \Big\},   \nonumber\\
i_1 &= \max\Big\{i\in [0,i_2):\, \# (E\cap (t_i, t_{i_2})) \le m_i+m_{i_2}-2+\sum_{l=i+1}^{i_2-1}2m_l \}.\label{eqn:n1}
\end{align}
Then we have
\begin{align}
  &   i_1  \le i_2-2,  \label{eqn:n1n2:1}\\
  &  \# (E\cap (t_{i_1}, t_{i_1+1})) = \# (E\cap (t_{i_1}, t_{i_1+1}]) = m_{i_1}+m_{i_1+1}-1, \label{eqn:n1 n11}  \\
  &  \# (E\cap (t_{i_2-1}, t_{i_2})) = \# (E\cap [t_{i_2-1}, t_{i_2})) = m_{i_2-1}+m_{i_2}-1, \label{eqn:n1n2:3} \\
  &  \#(E\cap [t_i,t_{i+1})) = \#(E\cap (t_i,t_{i+1}]) = m_i+m_{i+1}, \,\, i_1+1\le i\le i_2-2,\label{eqn:n1n2:4} \\
  &  E\cap(t_{i_1},t_{i_2})\cap \{t_i:\,i\in\bbZ\} = \emptyset. \label{eqn:n1n2:5}
\end{align}
Moreover, we see from the definitions of $i_1, i_2$ and (\ref{eqn:n1 n11}) that
\begin{align}
  \# (E\cap (t_i,t_{i+2})) &\ge m_i+2m_{i+1}+m_{i+2}-1,  & 0\le i\le i_1-1,  \label{eqn:nn2}\\
 \# (E\cap (t_i,t_{i_1}]) &\ge m_i + \sum_{l=i+1}^{i_1-1} 2m_l + m_{i_1}  , & 0\le i\le i_1-1.\label{eqn:nn1}
\end{align}
It follows 
that
if $i_1>1$, then
\begin{eqnarray*}
 \#(E\cap [t_0, t_{i_2}))
   &= &  \#(E\cap [t_0, t_1]) + \#(E\cap (t_1, t_{i_1}])    +  \#(E\cap (t_{i_1}, t_{i_2}))\\
   &\ge& m+m_{i_2}-2+\sum_{l=1}^{i_2-1}2m_l.
\end{eqnarray*}
A simple computation shows that the above inequality is also true if $i_1\le 1$.

By (\ref{eqn:n}), there is some $E_2\subset E$ such that
\begin{eqnarray}
  &&E\cap[t_{i_2}, t_k]\subset E_2, \label{eqn:E12:1}\\
  &&\#(E_2 \cap [t_0, t_{i_2}) ) = \sum_{l=1}^{i_2} m_l -1 ,  \label{eqn:E12:2}\\
  &&\#(E_2\cap(t_i,t_{i+1})) = m_{i+1}, \qquad 0 \le i\le i_2-2,\label{eqn:E12:3} \\
  &&\#(E_2\cap(t_{i_2-1},t_{i_2})) = m_{i_2}-1.\label{eqn:E12:4}
\end{eqnarray}
Let $E_1 = E\setminus E_2$.
Define
\begin{eqnarray}
  A_1 &=& [Q_{i,l}(x_n)]_{x_n\in E_1, (i,l)\in I_{0,k}}, \label{eqn:A1}\\
  A_2 &=& [Q_{i,l}(x_n)]_{x_n\in E_2, (i,l)\in I_{0,k}}. \label{eqn:A2}
\end{eqnarray}
Note that the first $\sum_{l=1}^{i_2}m_l$ columns of $A_2$ has the form
$\begin{pmatrix}
A_{21} \\
0
\end{pmatrix}$,
where $A_{21}$ is a $(\sum_{l=1}^{i_2}m_l-1)\times(\sum_{l=1}^{i_2}m_l)$ matrix which can be written as
\[
 A_{21}=
  \begin{pmatrix}
   * & * & *  & ?  & ?  & ? & ?  & \ldots \\
   * & * & *  & ?  & ?  & ? & ?  & \ldots \\
   0 & 0 & *  & *  & *  & * & ?  & \ldots \\
   0 & 0 & *  & *  & *  & * & ?  & \ldots \\
   0 & 0 & *  & *  & *  & * & ?  & \ldots \\
   0 & 0 & 0  & 0  & 0  & * & *  & \ldots \\
   0 & 0 & 0  & 0  & 0  & * & *  & \ldots \\
     &   &   & \ldots &  \\
  \end{pmatrix}.
\]
Here '$*$' stands for a non-zero entry,  '$?$' stands for uncertainty, and we set $m_1=2$ and $m_2=3$ for simplicity.
By Proposition~\ref{prop:Sch}, every $(\sum_{l=1}^{i_2}m_l-1)\times(\sum_{l=1}^{i_2}m_l-1)$ submatrix of $A_{21}$ is invertible. Hence there exists
some  $ c' \in \bbR^{N_{0,k}}$, for which none of the first $\sum_{l=1}^{i_2}m_l$ entries is zero and the rest entries are zeros, such that
\[
  A_2 c' = 0.
\]

On the other hand, since $E_1\subset [t_0, t_{i_2})$, we have $A_1 = (A_{11}, 0)$,
where $A_{11}$ has $m+1+\sum_{l=1}^{i_2-1}m_l$ columns and at least
$m-1+\sum_{l=1}^{i_2-1}m_l$ rows.
Note that
\[
\#(E_1\cap (t_{i_1}, t_{i_1+1}]) = m_{i_1}-1
\]
and
\[
  \#(E_1\cap (t_i, t_{i+1}]) = m_i \,\, \mathrm{for}\, i_1+1\le i\le i_2-1.
\]
We have $A_{11} = \begin{pmatrix}
? & 0 \\
? & A_{111}
\end{pmatrix}$,
where $A_{111}$ is a $(\sum_{l=i_1}^{i_2-1}m_l-1)\times (\sum_{l=i_1}^{i_2-1}m_l)$ matrix.
We see from the above arguments that $A_{111}$ has the following form,
\[
  A_{111} =
  \begin{pmatrix}
 B_{i_1}  &    0        & 0          & \ldots &   0 \\
 ?        & B_{i_1+1}   & 0          & \ldots &   0 \\
 ?        & ?           & B_{i_1+2}  & \ldots &   0 \\
          &\ldots       &            & \ldots &   0 \\
 ?        & ?           &  ?         & \ldots & B_{i_2-1} \\
  \end{pmatrix},
\]
where $B_{i_1}$ is an $(m_{i_1}-1)\times m_{i_1}$ matrix
and $B_i$ is an $m_i\times m_i$ matrix for $i_1+1\le i\le i_2-1$.
Hence there exists some real vector $\tilde c\ne 0$, which has
$\sum_{l=i_1}^{i_2-1}m_l$ entries, such that
\[
  A_{111}\tilde c=0.
\]
Note that all entries of $B_i$ are nonzeros, $i_1\le i\le i_2-1$.
Moreover, the $(p,q)$-entry of $A_{111}$ is also non-zero,
where $\sum_{l=i_1}^j m_l \le p\le \sum_{l=i_1}^{j+1} m_l-1$,
$q = \sum_{l=i_1}^j m_l$, $i_1\le j\le i_2-2$.
By Proposition~\ref{prop:Sch},
every $(\sum_{l=i_1}^{i_2-1}m_l-1)\times (\sum_{l=i_1}^{i_2-1}m_l-1)$ submatrix of $A_{111}$ is non-singular.
Hence no  entry of $\tilde c $ is zero.
Consequently, there is some $c= (0,\ldots, 0, \tilde c^t, 1, \ldots, 1)^t \in\bbR^{N_{0,k}}$,
where the last $\sum_{l=i_2}^{k-1}m_l$ entries of $c$ are equal to $1$,
such that
\[
  A_1 c = 0.
\]
Again, we assume that (\ref{eqn:cc1}) is true.

Let $f_1$ and $f_2$ be defined by (\ref{eqn:fg}).
Since none of the first $\sum_{l=1}^{i_2}m_l$ entries of $c'$ is zero,
for $1\le n\le i_2-1$,
there exists some index pair $(i,l)$ such that
$(t_{n-1}, t_{n+1}) \subset  \supp Q_{i,l}$ and
$c_{i,l} \pm c'_{i,l} \ne 0$.
Similarly, for $i_1+1\le n\le k-1$,
there exists some index pair $(i,l)$ such that
$(t_{n-1}, t_{n+1}) \subset  \supp Q_{i,l}$ and
$c_{i,l} \pm c'_{i,l} \ne 0$.
Since $i_1+1\le i_2-1$, we see from Lemma~\ref{Lem:sep} that
neither $f_1$ nor $f_2$ is separable.
With similar arguments as before we get a contradiction. Hence (\ref{eqn:PR:4}) is true.

{\color{blue}
Next we prove (\ref{eqn:PR:2}).}
Assume on the contrary that (\ref{eqn:PR:2}) is false.
Since (\ref{eqn:PR:2}) is true for $n=1$,
there is some $n_0\ge 2$ such that
$\#(E\cap [t_0, t_{n_0})) \le m+m_{n_0}-1+\sum_{l=1}^{n_0-1}2m_l$
 and $\#(E\cap [t_0, t_n)) \ge m+m_n+\sum_{l=1}^{n-1}2m_l$  for $1\le n\le n_0-1$.
By (\ref{eqn:PR:4}), $\#(E\cap (t_{n_0-1}, t_{n_0}) ) \ge m_{n_0-1} + m_{n_0} - 1$.
Hence $\#(E\cap [t_0, t_{n_0})) = m+m_{n_0}-1+\sum_{l=1}^{n_0-1}2m_l$
and $t_{n_0-1}\not\in E$.

As in the previous  arguments, we can split $E$ into two subsequences $E_1$ and $E_2$ which satisfy (\ref{eqn:E12:1})
- (\ref{eqn:E12:4}) with $i_2$ being replaced by $n_0$.
Define $A_1$ and $A_2$ by (\ref{eqn:A1}) and (\ref{eqn:A2}), respectively.
Then
 there exists
some  $ c' \in \bbR^{N_{0,k}}$, for which none of the first $\sum_{l=1}^{n_0}m_l$ entries is zero and all the rest entries are zeros, such that
\[
  A_2 c' = 0.
\]

On the other hand, since $\#E_1 = m+\sum_{l=1}^{n_0-1}m_l$,
by Proposition~\ref{prop:local sampling}, $E_1$ is not a linear sampling sequence for $\Vm|_{[t_0, t_{n_0}]}$.
Hence there is some $c\in\bbR^{N_{0,k}}\setminus \{0\}$, for which the last $\sum_{l=n_0}^{k-1} m_l$ entries are zeros, such that
\[
  A_1 c = 0.
\]
Again, we assume that (\ref{eqn:cc1}) holds
and define $f_1$ and $f_2$ by (\ref{eqn:fg}).
To show that (\ref{eqn:PR:2}) is true, it suffices to prove that both $f_1$ and $f_2$ are nonseparable.

Assume that $f_1$ is separable.
Observe that for the first $\sum_{l=1}^{n_0} m_l$ index pairs $(i,l)$,
$c_{i,l} + c'_{i,l}\ne0$.
By Lemma~\ref{Lem:sep}, there is some $p\ge n_0$ such that
$c_{i,l} + c'_{i,l} = 0$ whenever $(t_{p-1}, t_{p+1}) \subset \supp Q_{i,l}$.
Moreover, there is some $(i,l)$ with $i\ge p$ such that
$c_{i,l} + c'_{i,l} \ne 0$, which is impossible since $c_{i,l}=c'_{i,l}=0$ whenever $i\ge n_0$.
Similarly we can prove that $f_2$ is non-separable. This proves (\ref{eqn:PR:2}).
And  (\ref{eqn:PR:3}) can be proved similarly.

\textcolor{blue}{Finally, we prove (\ref{eqn:PR:1})}.
Assume on the contrary that $\#E \le 2m + \sum_{l=1}^{k-1}2m_l$.
To get a contradiction, we only need to consider the case of
$\#E = 2m + \sum_{l=1}^{k-1}2m_l$.

Take some $E_1\subset E$ such that
$\#E_1 =  m + \sum_{l=1}^{k-1} m_l$ and
\begin{equation}\label{eqn:s:1}
\#(E_1\cap (t_n,t_{n+1}))\ge   m_{n+1}, \qquad 0\le n\le k-1.
\end{equation}
We see from (\ref{eqn:PR:4}) that such $E_1$ exists.
Let $E_2=E\setminus E_1$.
Define $A_1$ and $A_2$ by (\ref{eqn:A1}) and (\ref{eqn:A2}), respectively.
Since $\rank(A_1)< N_{0,k}$,
there is some $c\in \bbR^{N_{0,k}}\setminus\{0\}$ such that
\[
  A_1 c = 0.
\]
We conclude that the function $g = \sum_{(i,l)\in I_{0,k}} c_{i,l}Q_{i,l}$ is not separable.
That is, there is not an integer $n\in [1,k-1]$ which satisfies the following two conditions,
\begin{enumerate}
\item[(P1)]  $c_{i,l}=0$ whenever $(t_{n-1}, t_{n+1})\subset \supp Q_{i,l}$,
\item[(P2)] there exist index pairs $(i_1,l_1)$, $(i_2,l_2)\in I_{0,k}$
such that $\supp Q_{i_1,l_1} \subset (-\infty, t_n)$,
$\supp Q_{i_2,l_2} \subset (t_n, \infty)$,
and $c_{i_p,l_p}\ne 0$ for $p=1,2$.
\end{enumerate}

Denote $r_n = \#(E_1\cap [t_{n-1},t_n))$.
We see from the choice of $E_1$ that $A_1$ is of the following structure,
\[
A_1 = \begin{pmatrix}
A_{1,1} \\
\vdots \\
A_{k,1}
\end{pmatrix},
\]
where $A_{n,1}$ has $r_n \ge m_n$ rows
and for each row, the first $\sum_{l=1}^{n-1} m_l$ entries are zeros,
the next $m$  or $m+1$ entries are non-zeros (it might be $m$ only for the first row of $A_{n,1}$), and the rest are zeros.
In other words, $A_1$ is of the following structure,

\begin{center}
\begin{tikzpicture}[scale=0.7]
\draw (0,0) rectangle (4,1);
\draw (1,-1.5) rectangle (5,0);
\draw (3,-2.3) rectangle (7,-1.5);
\draw (4,-3) node {$\cdots$};
\draw (5,-5) rectangle (9,-4);
\draw (-2,0.5) node {$A_{1,1}:$};
\draw (-2,-0.8) node {$A_{2,1}:$};
\draw (-2,-2) node {$A_{3,1}:$};
\draw (-2,-4.5) node {$A_{k,1}:$};
\draw (2,0.5) node {$\tilde A_{1,1}$};
\draw (3,-0.8) node {$\tilde A_{2,1}$};
\draw (5,-2) node {$\tilde A_{3,1}$};
\draw (7,-4.5) node {$\tilde A_{k,1}$};
\end{tikzpicture}
\end{center}
where $A_{n,1} = (0 \ \tilde A_{n,1}\ 0)$ and $\tilde A_{n,1}$ is an
$r_n \times (m+1)$ matrix.
Since $A_1 c=0$, we have
\begin{equation}\label{eqn:cn}
   \tilde A_{n,1} c^{(n)} = 0,
\end{equation}
where  $c^{(n)}\in\bbR^{m+1}$ whose entries consist  of $c_{i,l}$
for which $(t_{n-1},t_n) \subset \supp Q_{i,l}$.

Assume that  there exists an integer $n\in [1,k-1]$ such that $c_{i,l}=0$
whenever $(t_{n-1}, t_{n+1})\subset \supp Q_{i,l}$.
The the last $(m+1-m_n)$ entries of $c^{(n)}$ are zeros.
On the other hand, we see from Proposition~\ref{prop:Sch}
that the submatrix of $\tilde A_{n,1}$ consisting of the first $m_n$ columns is of rank $m_n$.
Now it follows from (\ref{eqn:cn}) that $c^{(n)}=0$. Consequently,
$c_{i,l}=0$ whenever $(t_{n-1}, t_n)\subset \supp Q_{i,l}$.
Therefore,
$c_{i,l}=0$
whenever $(t_{n-2}, t_{n})\subset \supp Q_{i,l}$.
By induction, we have
\[
  c_{i,l}=0,\qquad \mathrm{if}\quad \supp Q_{i,l}\cap  (t_0,t_n)\ne \emptyset.
\]
In other words,  (P2)  is false. Hence $g$ is nonseparable.

On the other hand, since $\#(E\setminus E_1) = \#E_1 < N_{0,k} $, the equation
\[
  A_2 c' = 0
\]
has a non-zero solution.
Again, we assume that (\ref{eqn:cc1}) holds. With similar arguments we   get a contradiction.
This completes the proof of the necessity.
\end{proof}

To prove the sufficiency of Theorem~\ref{Thm:phaseless:local}, we first present some  preliminary results.

\begin{Lemma}\label{Lem:L3a}
Let $p_1<q_1\le p_2<q_2$ be integers and $E\subset [p_1, q_2]$ be a sequence consisting of distinct points such that
\[
  \# (E\cap(t_{n_1}, t_{n_2}))      \ge m_{n_1} + m_{n_2}-1+ \sum_{l=n_1+1}^{n_2-1} 2 m_l
\]
whenever  $p_1\le n_1<n_2\le q_2$.
Suppose that  $E_1\cup E_2=E$, $E_1\cap  E_2=\emptyset $ and  that
$\#(E_i\cap[t_{p_i},t_{q_i}])\ge m+1+\sum_{l=p_i+1}^{q_i-1} m_l$, $i=1,2$.
Then there exist integers
$p'_1<q'_1=p'_2<q'_2$ in $[p_1,q_2]$ such  that
$E_i\cap[t_{p'_i}, t_{q'_i}]$  is a linear sampling sequence for   $\Vm|_{[t_{p'_i}, t_{q'_i}]}$, $i=1,2$.
\end{Lemma}

\begin{proof}
By Lemma~\ref{Lem:samp},
for $i=1,2$,
there exist integers $r_i, s_i\in [p_i, q_i]$ such that
$E_i\cap[t_{r_i}, r_{s_i}]$ is a linear sampling sequence for
 $\Vm|_{[t_{r_i}, t_{s_i}]}$.
Without loss of generality, we assume that $(r_i,s_i)= (p_i,q_i)$.

Recall that $q_1\le p_2$.
Let $(p'_1,q'_1)$ be a pair of integers such that
$E_1\cap [t_{p'_1},t_{q'_1}]$ is a linear sampling sequence for
$\Vm|_{[t_{p'_1},t_{q'_1}]}$ and
$q'_1$ is the maximum of all integers
$q\in [p_1, p_2]$ for which
there is some $p\in [p_1,q)$ such that $E_1\cap [t_p,t_q]$ is a linear sampling sequence for
$\Vm|_{[t_p,t_q]}$.
And Let $(p'_2,q'_2)$ be a pair of integers such that
$E_2\cap [t_{p'_2},t_{q'_2}]$ is a linear sampling sequence for
$\Vm|_{[t_{p'_2},t_{q'_2}]}$ and
 $p'_2$ is the minimum of all integers
$p\in [q'_1, q_2]$ for which there is some integer $q\in (p, q_2]$ such that $E_2\cap [t_p,t_q]$ is a linear sampling sequence for
$\Vm|_{[t_p,t_q]}$.
Then we have $q_1 \le q'_1\le  p'_2\le p_2$.

To complete the proof, it suffices to show that $q'_1=p'_2$.

Assume on the contrary that $q'_1<p'_2$.
Note that
\[
  \#(E\cap (q'_1, p'_2)) \ge  m_{q'_1}+m_{p'_2}-1+\sum_{l=q'_1+1}^{p'_2-1} 2m_l.
\]
We have either
\begin{equation}\label{eqn:pq1}
  \#(E_1\cap (q'_1, p'_2)) \ge  \sum_{l=q'_1}^{p'_2-1}m_l
\end{equation}
or
\begin{equation}\label{eqn:pq2}
  \#(E_2\cap (q'_1, p'_2)) \ge  \sum_{l=q'_1+1}^{p'_2}m_l.
\end{equation}
Without loss of generality, we assume that (\ref{eqn:pq1}) is true.
Let
\begin{equation}\label{eqn:dj2}
n_0 = \min\{ n>q'_1:\, \#(E_1\cap (t_{q'_1}, t_n)) \ge  \sum_{l=q'_1}^{n-1} m_l   \}.
\end{equation}
Then we have $q'_1 < n_0 \le p'_2$. For $q'_1<n<n_0$, since
\[
   \#(E_1\cap (t_{q'_1}, t_n)) \le  \sum_{l=q'_1}^{n-1} m_l -1 ,
\]
we have
\[
   \#(E_1\cap [t_n, t_{n_0})) \ge  \sum_{l=n}^{n_0-1} m_l+1.
\]
Hence for $q'_1 \le n<n_0$,
\begin{equation}\label{eqn:j2j}
   \#(E_1\cap (t_n, t_{n_0})) \ge  \sum_{l=n}^{n_0-1} m_l.
\end{equation}
It follows that
\begin{eqnarray*}
 \#(E_1\cap [t_{p'_1}, t_{n_0}])
   &= &  \#(E_1\cap [t_{p'_1}, t_{q'_1}]) +  \#(E_1\cap (t_{q'_1}, t_{n_0}])  \\
   &\ge& m + 1 + \sum_{l=p'_1+1}^{n_0-1} m_l.
\end{eqnarray*}

On the other hand, since $E_1\cap[t_{p'_1}, t_{q'_1}]$ is a linear sampling sequence for $\Vm|_{[t_{p'_1}, t_{q'_1}]}$,
for $p'_1\le n < q'_1$, we have
\begin{align*}
   \#(E_1\cap (t_n, t_{n_0})) &=    \#(E_1\cap (t_n, t_{q'_1}])  +    \#(E_1\cap (t_{q'_1}, t_{n_0}))  \\
   &\ge \sum_{l=n}^{n_0-1} m_l.
\end{align*}
In other words, (\ref{eqn:j2j}) is true for all $p'_1 \le n   < n_0$.

We see from the choice of $q'_1$ that $E_1\cap [t_{p'_1}, t_{n_0}]$ is not a linear sampling sequence
for $\Vm|_{[t_{p'_1}, t_{n_0}]}$.
By Proposition~\ref{prop:local sampling}, there are two cases.

(a)\,\, There exist integers $n_1,n_2\in [p'_1, n_0]$ such that $n_1<n_2$ and
$\#(E_1\cap(t_{n_1}, t_{n_2})) \le  \sum_{l=n_1}^{n_2} m_l-m-2$.

Recall that $E_1\cap [t_{p'_1}, t_{q'_1}]$ is a linear sampling sequence
for $\Vm|_{[t_{p'_1}, t_{q'_1}]}$.
We have $n_2>q'_1$.
First, we assume that $n_1 < q'_1$.
Since $\#(E_1\cap(t_{n_1}, t_{q'_1}]) \ge \sum_{l=n_1}^{q'_1-1} m_l $,
we have
\[
  \#(E_1\cap(t_{q'_1}, t_{n_2})) \le \sum_{l=q'_1}^{n_2} m_l-m-2.
\]
It follows from (\ref{eqn:j2j}) that
\[
  \# (E_1\cap (t_{n_2}, t_{n_0})) \ge  \sum_{l=n_2+1}^{n_0-1} m_l+m+1.
\]
By Lemma~\ref{Lem:samp}, there exist integers
$n'_1, n'_2\in [n_2, n_0]$  such that $n'_1<n'_2$ and   $E_1\cap [t_{n'_1}, t_{n'_2}]$ is linear sampling sequence for
$\Vm|_{[t_{n'_1}, t_{n'_2}]}$, which contradicts with the choice of $q'_1$.

Next we consider the case of $n_1\ge q'_1$.
We see from (\ref{eqn:dj2}) that
\begin{eqnarray*}
\#(E_1\cap (t_{q'_1},t_{n_2}))
&=&   \#(E_1\cap (t_{q'_1},t_{n_1})) + \#(E_1\cap [t_{n_1},t_{n_2}))\\
&\le& \sum_{l=q'_1}^{n_2} m_l - m -2.
\end{eqnarray*}
Similar to the previous case we get a contradiction.

(b).\,\,  There exists some integer $n\in [p'_1+1, n_0]$ such that
$\#(E_1\cap [t_{p'_1}, t_n )) \le \sum_{l=p'_1+1}^n m_l-1$.

In this case, we have $n>q'_1$ and
\begin{eqnarray*}
\#(E_1\cap(t_{q'_1},t_n)) &=& \#(E_1\cap[t_{p'_1},t_n)) - \#(E_1\cap[t_{p'_1}, t_{q'_1}]) \\
&\le& \sum_{l=q'_1}^{n} m_l - m -2.
\end{eqnarray*}
Again,
similar to the previous case we get a contradiction.  This completes the proof.
\end{proof}

The following is a simple application of Lemma~\ref{Lem:L3a}.

\begin{Lemma}\label{Lem:L3}
Let $E\subset [t_0, t_k]$ be a sequence of distinct points which meets (\ref{eqn:PR:1}) - (\ref{eqn:PR:4}).
Suppose that  $E_1\cup E_2=E$, $E_1\cap  E_2=\emptyset $, $\#E_1\le  \# E_2$, and that  $E_2$  is not a linear sampling sequence for
$\Vmk$.
Then there exist integers
 $p'_1,q'_1,p'_2,q'_2\in [0,k]$ such that
 $p'_1<q'_1$, $p'_2<q'_2$,  $[p'_1, q'_1]$ and $[p'_2, q'_2]$ have and only have one common point,
 and
$E_i\cap[p'_i, q'_i]$  is a linear sampling sequence for   $\Vm|_{[t_{p'_i}, t_{q'_i}]}$, $i=1,2$.
\end{Lemma}

\begin{proof}
Since $\# E_2 \ge \#E_1$, we see from (\ref{eqn:PR:1}) that
\begin{equation}\label{eqn:aba:e1}
\#E_2 \ge m+1+\sum_{l=1}^{k-1} m_l.
\end{equation}
Note that $E_2$  is not a linear sampling sequence for $\Vmk$.
By Proposition~\ref{prop:local sampling}, there are three cases.

(i) There is some integer $i\in [1, k]$ such that $\#(E_2\cap[t_0, t_i)) \le \sum_{l=1}^i m_l -1 $.

In this case, we see from (\ref{eqn:PR:2}) that
\begin{eqnarray*}
 \#(E_1\cap[t_0, t_i)) &=& \#(E\cap[t_0, t_i)) - \#(E_2\cap[t_0, t_i)) \\
   &\ge& m+1+\sum_{l=1}^{i-1}m_l.
\end{eqnarray*}
By Lemma~\ref{Lem:samp}, there exist integers $p_1, q_1 \in [0,i]$ such that $p_1<q_1$ and
$E_1\cap [t_{p_1}, t_{q_1}]$ is a linear sampling sequence for $\Vm|_{[t_{p_1}, t_{q_1}]}$.

On the other hand, note that
\begin{eqnarray*}
 \#(E_2\cap[t_i, t_k]) &=& \#(E_2\cap[t_0, t_k]) - \#(E_2\cap[t_0, t_i)) \\
   &\ge& m+2+\sum_{l=i+1}^{k-1} m_l.
\end{eqnarray*}
Using Lemma~\ref{Lem:samp} again, we get some integers $p_2, q_2\in [i, k]$  such that $p_2<q_2$
and $E_2\cap [t_{p_2}, t_{q_2}]$ is a linear sampling sequence for $\Vm|_{[t_{p_2}, t_{q_2}]}$.
Now the conclusion follows from Lemma~\ref{Lem:L3a}.

(ii) There is some integer $i\in [0,k-1]$  such that $\#(E_2\cap(t_i, t_k]) \le \sum_{l=i}^{k-1} m_l-1$.

Similarly to the first case we can prove the conclusion.

(iii) There exist integers $i,j\in [0,k]$ such that $i<j$  and $\#(E_2\cap(t_i, t_j)) \le \sum_{l=i}^j m_l-m-2$.

In this case, we see from (\ref{eqn:PR:4}) that
\[
  \#(E_1\cap (t_i, t_j)) \ge m+1+\sum_{l=i+1}^{j-1} m_l.
\]
Using Lemma~\ref{Lem:samp} again, we  get integers $p_1, q_1 \in [i,j]$ such that $p_1<q_1$ and
$E_1\cap [t_{p_1}, t_{q_1}]$ is a linear sampling sequence for $\Vm|_{[t_{p_1}, t_{q_1}]}$.

On the other hand, we see from  (\ref{eqn:aba:e1}) that
either
\[
  \#(E_2\cap [t_0, t_i]) \ge m+1+\sum_{l=1}^{i-1} m_l
\]
or
\[
  \#(E_2\cap [t_j, t_k]) \ge m+1+\sum_{l=j+1}^{k-1} m_l.
\]
Consequently, there exist integers $p_2<q_2$ such that
$[p_2,q_2]\subset [0,i]$ or $[p_2,q_2]\subset [j,k]$
and
$E_2\cap [t_{p_2}, t_{q_2}]$ is a linear sampling sequence for $\Vm|_{[t_{p_2}, t_{q_2}]}$.
Again, the conclusion follows from Lemma~\ref{Lem:L3a}.
\end{proof}

The sufficiency  of Theorem~\ref{Thm:phaseless:local}   is a consequence of the following lemma, which says that if $f_1\pm f_2 \ne 0$
and $|f_1|$ and $|f_2|$ are identical on a sequence $E$ which satisfies (\ref{eqn:PR:1}) - (\ref{eqn:PR:4}),
then   $f_1$ and $f_2$ are  uniformly separable.

\begin{Lemma}\label{Lem:L4}
Suppose that $E=\{x_n:\, 1\le n\le N\}$ is a sequence of distinct points which satisfies (\ref{eqn:PR:1}) - (\ref{eqn:PR:4}).
Let $f_p = \sum_{(i,l)\in I_{0,k}} c^{(p)}_{i,l} Q_{i,l}\in \Vm$, $p=1,2$. Suppose that
\begin{equation}\label{eqn:f12i}
|f_1(x_n)| = |f_2(x_n)|, \qquad 1\le n\le N.
\end{equation}
Then $|f_1(x)|=|f_2(x)|$ on $[t_0, t_k]$. Moreover,
if $f_1\ne \pm f_2$,  then   $f_1$ and $f_2$ are uniformly separable in the sense that there is some integer $n_0\in (0,k)$ such that
$f_p|_{[t_0, t_{n_0}]}\ne 0$,
$f_p|_{[t_{n_0}, t_k]}\ne 0$,
and
\begin{equation}\label{eqn:cn0}
 c^{(p)}_{i,l}=0,\qquad \mathrm{if} \,\, (t_{n_0-1}, t_{n_0+1}) \subset \supp Q_{i,l}, \quad p=1,2.
\end{equation}
\end{Lemma}

\begin{proof}
Suppose that $f_1\pm f_2\ne 0$.
Split $E$ into two subsequences $E_1$ and $E_2$ such that
\[
  E_1 = \{x_n:\, f_1(x_n) = -f_2(x_n)\}\quad \mathrm{and}\quad  E_2 = \{x_n:\, f_1(x_n) = f_2(x_n)\}.
\]
Since $\#E \ge 2m+1+\sum_{l=1}^{k-1} 2m_l$, without loss of generality, we assume that
$\# E_2 \ge m+1+\sum_{l=1}^{k-1} m_l$.

Since $f_1\ne f_2$ and they are identical on $E_2$,
$E_2$ is not a sampling sequence for $\Vmk$.
We see from   Proposition~\ref{prop:local sampling} that $k\ge 2$.
We prove the conclusion with induction over  $k$.

First, we consider the case of $k=2$.
Since $E_2$ is not a sampling sequence for $\Vmk$, we see from Proposition~\ref{prop:local sampling}
that either $\#(E_2\cap [t_0, t_1))\le m_1-1$ or $\#(E_2\cap (t_1, t_2]) \le m_1-1$.
Without loss of generality, assume that $\#(E_2\cap (t_1, t_2]) \le m_1-1$.

Since  $f_1 \pm f_2 \ne 0$, there are some $c, c'\in \bbR^{N_{0,k}}\setminus\{0\}$ such that
\begin{equation} \label{eqn:f12}
\left\{
\begin{aligned}
f_1(x) + f_2(x) &= \sum_{(i,l)\in I_{0,k}} c_{i,l}Q_{i,l}, \\
f_1(x) - f_2(x) &= \sum_{(i,l)\in I_{0,k}} c'_{i,l}Q_{i,l}.
\end{aligned}\right.
\end{equation}
Since  $\#(E_2\cap[t_0, t_1]) =
\#(E_2\cap[t_0, t_2]) - \#(E_2\cap(t_1, t_2])
 \ge 2+m$,
$E_2\cap[t_0, t_1]$ is a linear sampling sequence for $\Vm|_{[t_0, t_1]}$.
We see from the definition of $E_2$ that $f_1$ and $f_2$ are identical on $E_2$.
Hence they are identical on  $[t_0, t_1]$. Therefore,
\[
  c'_{i,l}=0,\qquad   \mathrm{if}\, (t_0,t_1)\subset \supp Q_{i,l}.
\]
But $f_1\ne  f_2$. Hence
there exist some $(i_1,l_1)\in I_{0,2}$ such that
$(t_0,t_1)\not\subset \supp Q_{i_1,l_1}$ and
$c'_{i_1,l_1}\ne 0$.

On the other hand, since    $\#(E_1\cap(t_1, t_2]) = \#(E\cap(t_1, t_2]) - \#(E_2\cap(t_1, t_2]) \ge m+1$,
$E_1\cap [t_1,t_2]$ is a linear sampling sequence for $\Vm|_{[t_1,t_2]}$.
Now we see from  $f_1(x_n) = -f_2(x_n)$ for $x_n\in E_1$ that
$f_1(x)=-f_2(x)$ on $[t_1, t_2]$ and
\[
  c_{i,l}=0,\qquad   \mathrm{if}\, (t_1,t_2)\subset \supp Q_{i,l}.
\]
But $f_1\ne -f_2$. Hence
there exist some $(i_2,l_2)\in I_{0,2}$ such that
$(t_1,t_2)\not\subset \supp Q_{i_2,l_2}$ and
$c_{i_2,l_2}\ne 0$.

Summing up the above arguments we get that
$ c_{i,l} =c'_{i,l} =0$ if $(t_0,t_2)\subset \supp Q_{i,l}$
and there exist $(i_1,l_1), (i_2,l_2)\in I_{0,k}$
such that $(t_0,t_1) \subset \supp Q_{i_2,l_2}$,
$(t_1,t_2) \subset \supp Q_{i_1,l_1}$,
and
\[
  c_{i_p, l_p} \pm c'_{i_p, l_p} \ne 0,\qquad p=1,2.
\]
By Lemma~\ref{Lem:sep}, both $f_1$ and $f_2$ are separable.
By setting $n_0=1$, we get the conclusion as desired.

Now suppose that  the conclusion is true for $2\le k\le k_0$, where $k_0\ge 2$.
Let us consider the case of $k=k_0+1$.

Recall that $E_2$ is not a linear sampling sequence for $\Vmk$.
We see from  Lemma~\ref{Lem:L3} that there exist integers
$p_1,q_1, p_2, q_2\in [0, k]$ such that
$p_1<q_1$, $p_2<q_2$,  $[p_1, q_1]$ and $[p_2, q_2]$ have and only have one common point,
 and
$E_1\cap[t_{p_1}, t_{q_1}]$  and
$E_2\cap[t_{p_2}, t_{q_2}]$
are linear sampling sequences for   $\Vm|_{[t_{p_1}, t_{q_1}]}$
and $\Vm|_{[t_{p_2}, t_{q_2}]}$, respectively.
Hence $f_1 = -f_2$ on $[t_{p_1}, t_{q_1}]$
and  $f_1 = f_2$ on $[t_{p_2}, t_{q_2}]$.
Without loss of generality, we assume that $q_1=p_2$.

Take some $F\subset [t_{q_1-1}, t_{q_1}]\setminus E$
and $F' \subset [t_{p_2}, t_{p_2+1}]\setminus E$
 such that $\#F=\#F' =m$.
Let $\tilde E = F \cup (E\cap [t_0, t_{q_1}])$
and $\tilde E' = F' \cup (E\cap [t_{p_2}, t_k])$.
Then (\ref{eqn:PR:1}) -   (\ref{eqn:PR:4}) are true if we replace $(E, t_0, t_k)$
by $(\tilde E, t_0, t_{q_1})$ or $(\tilde E', t_{p_2}, t_k)$.
Since both $q_1$ and $k-p_2$ are less than $k$,
we see from the inductive assumption that
$\tilde E $  and $\tilde E'$ are phaseless sampling sequences for
$\Vm|_{[t_0, t_{q_1}]}$
and $\Vm|_{[t_{p_2}, t_k]}$,  respectively.
Hence $|f_1|$ and $|f_2|$ are identical on $[t_0, t_k]= [t_0, t_{q_1}]\cup [t_{p_2}, t_k]$.

It remains to show that
$f_1$ and $f_2$ are uniformly separable.

Take some $c,c'\in\bbR^{N_{0,k}}$ such that (\ref{eqn:f12}) holds.
Since $E_1\cap [t_{p_1}, t_{q_1}]$ is a linear sampling sequence for $\Vm|_{[t_{p_1}, t_{q_1}]}$ and
$f_1(x_n) = -f_2(x_n)$ for $x_n\in E_1$, we have
\begin{equation}\label{eqn:cni}
c_{i,l}=0,\qquad \mathrm{if}\, (t_{p_1}, t_{q_1})\cap \supp Q_{i,l} \ne \emptyset.
\end{equation}
Similarly we get
\[
c'_{i,l}=0,\qquad \mathrm{if}\, (t_{p_2}, t_{q_2})\cap \supp Q_{i,l} \ne \emptyset.
\]
Hence
\begin{equation}\label{eqn:cc3}
c_{i,l}=c'_{i,l} = 0,\qquad  \mathrm{if}\,  (t_{q_1-1}, t_{p_2+1})\subset  \supp Q_{i,l}.
\end{equation}
It follows that (\ref{eqn:cn0}) is true for $n_0 = q_1$.

Note that $E$ is a linear sampling sequence for $\Vmk$. We conclude that
neither $f_1$ nor $f_2$ is zero.
Otherwise, we see from $|f_1(x_n)|=|f_2(x_n)|$ for $x_n\in E$ that both are zeros,
which contradicts with the hypothesis $f_1 \pm f_2\ne 0$.
Hence there are $(i_1,l_1), (i_2, l_2) \in I_{0,k}$ such that
$c_{i_1,l_1}+c'_{i_1,l_1} \ne 0$ and $c_{i_2,l_2} - c'_{i_2,l_2} \ne 0$.
Without loss  of generality, assume that $(i_1, l_1) \le (i_2, l_2)$.
That is, $i_1<i_2$ or $i_1=i_2$ and $l_1>l_2$.

We see from (\ref{eqn:cc3}) that $ (t_{n_0-1}, t_{n_0+1}) \not\subset \supp Q_{i_r, l_r}$, $r=1,2$.
There are three cases.

(i).\, $\supp Q_{i_1,l_1}\subset (-\infty, t_{n_0})$ and $\supp Q_{i_2,l_2} \subset (t_{n_0}, \infty )$.

By (\ref{eqn:f12}),  $f_1 = \sum_{(i,l)\in I_{0,k}} \frac{c_{i,l}+c'_{i,l}}{2} Q_{i,l}$.
Hence  $f_1|_{[t_0, t_{n_0}]}\ne 0$.
Since $E\cap [t_0, t_{n_0}]$ is a linear sampling sequence for $\Vm|_{[t_0, t_{n_0}]}$,
we see from (\ref{eqn:f12i}) that $f_2|_{[t_0, t_{n_0}]}\ne 0$.
Similarly we can prove that $f_i|_{[t_{n_0}, t_k]} \ne 0$, $i=1,2$.
Hence $f_1$ and $f_2$ are uniformly separable.

(ii).\,
For any $(i_1,l_1), (i_2, l_2) \in I_{0,k}$ with
$c_{i_1,l_1}+c'_{i_1,l_1} \ne 0$ and $c_{i_2,l_2} - c'_{i_2,l_2} \ne 0$,
both $Q_{i_1,l_1}$ and $Q_{i_2,l_2}$ are supported in $(-\infty, t_{n_0})$.

In this case, we have
\[
  c_{i,l}=c'_{i,l}=0,\qquad \mathrm{if}\, \supp Q_{i,l}\cap (t_{n_0}, \infty ) \ne \emptyset.
\]
Hence
\[
  \left\{
\begin{aligned}
f_1   + f_2  &= \sum_{(i,l)\in I_{0,n_0}} c_{i,l}Q_{i,l}, \\
f_1  - f_2 &= \sum_{(i,l)\in I_{0,n_0}} c'_{i,l}Q_{i,l}.
\end{aligned}\right.
\]
Let $y_1, \ldots, y_m$ in $(t_{n_0-1},t_{n_0})\setminus E$ be $m$ distinct points
and set $\tilde E = (E\cap[t_0, t_{n_0}])\cup \{y_n:\,1\le n\le m\}$.
Then $\tilde E$ meets (\ref{eqn:PR:1}) - (\ref{eqn:PR:4}) if we
replace $(E, k)$ by $(\tilde E, n_0)$.

Let $\tilde E_1 = (E_1 \cap [t_0, t_{n_0}])\cup \{y_n:\,1\le n\le m\}$
and $\tilde E_2 = E_2\cap [t_0, t_{n_0}]$.
Then we have $\tilde E = \tilde E_1 \cup \tilde E_2$ and $f_1(x_n) = f_2(x_n)$ for $x_n\in \tilde E_2$.

On the other hand, note that $E_1\cap [t_{p_1}, t_{q_1}]$ is a linear sampling sequence for $\Vm|_{[t_{p_1}, t_{q_1}]}$
and $f_1+f_2$ is identical to zero on $E_1$, we have $(f_1+f_2)|_{[t_{p_1},t_{q_1}]} = 0$.
Hence $f_1(y_n) + f_2(y_n) = 0$ for $1\le n\le m$.
Therefore $(f_1+f_2)|_{\tilde E_1}=0$.
Since $n_0=q_1 < q_2\le k$,  we see from the inductive assumption that
$f_1$ and $f_2$ are uniformly separable.

(iii).\,
For any $(i_1,l_1), (i_2, l_2) \in I_{0,k}$ with
$c_{i_1,l_1}+c'_{i_1,l_1} \ne 0$ and $c_{i_2,l_2} - c'_{i_2,l_2} \ne 0$,
both $Q_{i_1,l_1}$ and $Q_{i_2,l_2}$ are supported in $(t_{n_0}, \infty)$.

Similarly to the previous case we can prove that
$f_1$ and $f_2$ are uniformly separable.

By induction, the conclusion is true for any $k\ge 2$. This completes the proof.
\end{proof}

\section{Linear Phaseless Sampling in Complex Spline Spaces}

\subsection{Characterization of Linear Phaseless Sampling Sequences}

The main result in this section is the following characterization of linear
phaseless sampling sequences for $|\Wm|^2_{[t_0, t_k]}$,
where $\Wm$ is defined similarly as $\Vm$ but with complex coefficients.

For simplicity, we set $M_k := km + N_{0,k} = 2m+1+\sum_{l=1}^{k-1}(m+m_l)$.

\begin{Theorem}\label{thm:main:2}
Let $k\ge 0$  be an integer
and $E=\{x_i:\, 1\le i\le K\} \subset [t_0, t_k]$  be a sequence  consisting of distinct points.
Then
$E$ is a   linear sampling sequence for $|\Wm|^2_{[t_0, t_k]}$  if and only if it  satisfies the followings,
\begin{align}
\#  E                     &\ge M_k,                                                   \label{eqn:ab:1} \\
\# (E\cap [t_0, t_i))     &\ge  \sum_{l=1}^i (m+m_l),       \quad   1\le i\le k,    \label{eqn:ab:2} \\
\# (E\cap (t_i, t_k]) &\ge  \sum_{l=i}^{k-1} (m+m_l),       \quad   0\le i\le k-1,      \label{eqn:ab:3} \\
\# (E\cap(t_i, t_j))      &\ge  \sum_{l=i}^j (m+m_l) - 2m-1,     \quad   0\le i<j\le k.  \label{eqn:ab:4}
\end{align}
\end{Theorem}

To prove this theorem, we need the following lemma, which itself is also interesting.

\begin{Lemma}\label{Lem:dimension}
Let $k\ge 1$ be an integer. Define
\begin{align}
\calP_k&=\{R_1Q_{i,l}:\,\supp Q_{i,l}\supset (t_0, t_1)\}  \label{eqn:F}\\
& \qquad \bigcup \{Q_{i,1}Q_{j,l}:\,0\le i\le k-1, Q_{j,l}\ne R_{i+1},  \supp Q_{j,l}\supset (t_i, t_{i+1}) \}
   \nonumber \\
& \qquad \bigcup \{R_{i+1}Q_{i,l}:\,1\le i\le k-1, 1\le l\le m_i \}.\nonumber
\end{align}
Then $\calP_k|_{[t_0,t_k]}$ is a basis for $\Lspan (|\Vm|^2_{[t_0, t_k]})$. Therefore,
\[
  \dim |\Vm|^2_{[t_0, t_k]}  = m(k+1)+1+\sum_{i=1}^{k-1} m_i = M_k.
\]
\end{Lemma}

\begin{proof}
 We prove the conclusion  by induction over $k$.

First, we consider the case of $k=1$.
In this case, for any $f\in \Vm$, $|f|^2_{[t_0, t_1]}$ is a polynomial of degree no greater than $2m$. Hence
$ \dim |\Vm|^2_{[t_0, t_1]}  \le 2m+1$.

On the other hand, suppose that there exist sequences of constants $\{c_{i,l}\}$ and $\{d_{j,l}\}$ such that
\begin{align}
&  \sum c_{i,l} R_1(x) Q_{i,l}(x)
  + \sum  d_{j,l} Q_{0,1}(x)Q_{j,l}(x) = 0,\quad x\in [t_0,t_1], \label{eqn:dim:e1}
\end{align}
where $Q_{i,l}$ and $Q_{j,l}$ run over all B-splines whose supports contain the interval $(t_0, t_1)$
and  $Q_{j,l} \ne R_1$.

It follows that $R_1(x)$ is a factor of  $\sum  d_{j,l} Q_{0,1}(x)Q_{j,l}(x) $.
Observe that $R_1(x) = a (t_1-x)^m$
and  $Q_{0,1}(x) = a' (x-t_0)^m$ for $x\in (t_0,t_1)$, where $a$ and $a'$ are constants.
Hence $R_1(x)$ and $Q_{0,1}(x)$  have no common factor but constants.
Therefore, there is some constant $c'$ such that
\[
  c'R_1 (x) = \sum  d_{j,l}  Q_{j,l}(x),  \qquad x\in [t_0,t_1].
\]
Recall that $Q_{j,l}\ne R_1$ in the above equation.
Since the $m+1$ B-splines whose supports contain the interval $(t_0, t_1)$ are linearly independent,
we see from the above equation that
\[
  d_{j,l} = 0.
\]
Using the linear independence of $Q_{i,l}$ we get that  $c_{i,l}=0$.
Hence $\calP_1$ is linearly independent.

Observe that for any $(i,l)\ne (i',l')$,
\[
  Q_{i,l}(x) Q_{i',l'}(x) = \frac{1}{2} \Big( (Q_{i,l}(x)+ Q_{i',l'}(x))^2 -(Q_{i,l}(x) - Q_{i',l'}(x))^2\Big).
\]
We have $\calP_1|_{[t_0, t_1]} \subset \Lspan (|\Vm|^2_{[t_0, t_1]})$.
Hence $\dim |\Vm|^2_{[t_0, t_1]}\ge 2m+1$. Therefore,
\[
  \dim |\Vm|^2_{[t_0, t_1]} = 2m+1
\]
and $\calP_1|_{[t_0, t_1]}$ is a basis for $\Lspan (|\Vm|^2_{[t_0, t_1]})$.

Now assume that
$\calP_{k-1}|_{[t_0, t_{k-1}]}$ is a basis for $\Lspan (|\Vm|^2_{[t_0, t_{k-1}]})$ for some $k\ge 2$.
Let us show that
$\calP_k|_{[t_0, t_k]}$ is a basis for $\Lspan (|\Vm|^2_{[t_0, t_k]})$.

First, we show that $\calP_k|_{[t_0, t_k]}$
 is linearly independent on $[t_0, t_k]$.
Assume that there are constants $c_{i,l}$,
$d_{i,j,l}$ and $c'_{i,l}$ such that
\begin{eqnarray}
&&\sum_{(i,l):\, \supp Q_{i,l}\supset (t_0, t_1)}  c_{i,l}R_1(x)Q_{i,l}(x) \label{eqn:dim:e2}\\
&&\qquad    + \sum_{i=0}^{k-1} \sum_{\substack{(j,l):\,Q_{j,l}\ne R_{i+1} \\  \supp Q_{j,l}\supset (t_i, t_{i+1}) }}
   d_{i,j,l} Q_{i,1}(x) Q_{j,l}(x)  \nonumber \\
&&\qquad    + \sum_{i=1}^{k-1} \sum_{l=1}^{m_i} c'_{i,l} R_{i+1}(x) Q_{i,l}(x) = 0,
  \qquad  x\in [t_0, t_k].  \nonumber
\end{eqnarray}
Since $Q_{k-1,l}(x)=0$ for $x<t_{k-1}$, we have
\begin{eqnarray}
&&\sum_{(i,l):\, \supp Q_{i,l}\supset (t_0, t_1)}  c_{i,l}R_1(x)Q_{i,l}(x) \nonumber \\
&&\qquad    + \sum_{i=0}^{k-2} \sum_{\substack{(j,l):\,Q_{j,l}\ne R_{i+1} \\  \supp Q_{j,l}\supset (t_i, t_{i+1}) }}
   d_{i,j,l} Q_{i,1}(x) Q_{j,l}(x)  \nonumber \\
&&\qquad    + \sum_{i=1}^{k-2} \sum_{l=1}^{m_i} c'_{i,l} R_{i+1}(x) Q_{i,l}(x) = 0,
  \qquad  x\in [t_0, t_{k-1}]. \nonumber
\end{eqnarray}
By the inductive assumption,
$\calP_{k-1}|_{[t_0, t_{k-1}]}$ is a basis for $\Lspan (|\Vm|^2_{[t_0, t_{k-1}]})$.
Hence all   coefficients in the above equation are zeros.
It follows from (\ref{eqn:dim:e2}) that
for $x\in [t_{k-1}, t_k]$,
\begin{equation}\label{eqn:cd1}
   \sum_{\substack{(j,l):\, Q_{j,l}\ne R_k \\  \supp Q_{j,l}\supset (t_{k-1}, t_k) }}
   d_{k-1,j,l} Q_{k-1,1}(x) Q_{j,l}(x)
   +     \sum_{l=1}^{m_{k-1}} c'_{k-1,l} R_k(x) Q_{k-1,l}(x) = 0.
\end{equation}
Since $Q_{k-1,1}(x) = a (x-t_{k-1})^m$
and $R_k(x) = a'(t_k-x)^m$ on $(t_{k-1}, t_k)$,
$Q_{k-1,1}(x)$ and $R_k(x)$ have no common factor but constants.
Hence
there exists some constant  $c'$ such that
\begin{equation}\label{eqn:cd2}
  \sum_{\substack{(j,l):\, Q_{j,l}\ne R_k \\  \supp Q_{j,l}\supset (t_{k-1}, t_k) }}
   d_{k-1,j,l} Q_{j,l}(x)
= c' R_k(x), \qquad x\in [t_{k-1}, t_k].
\end{equation}
Again, using the linear independence of $Q_{j,l}$ and $R_{k+1}$, we see that
all coefficients in (\ref{eqn:cd2}) and therefore coefficients in (\ref{eqn:cd1}) are zeros.
Hence
$\calP_k|_{[t_0, t_k]}$
is linearly independent on $[t_0, t_k]$.
Therefore,
\[
  \dim |\Vm|^2_{[t_0, t_k]} \ge \# \calP_k = m(k+1)+1+\sum_{i=1}^k m_i = M_k.
\]

It remains to show that
$\dim |\Vm|^2_{[t_0, t_k]} \le M_k$.
Suppose that
\[
  |f(x)|^2 = \sum_{n=0}^{2m} c_{i,n} (x-t_i)^n,\qquad x\in [t_i,t_{i+1}], 0\le i\le k-1.
\]
Since $f$ has up to $(m-m_i)$-th derivative at $t_i$ for $1\le i\le k-1$, we have
\begin{equation} \label{eqn:e:c}
   \sum_{n=p}^{2m}  \frac{n!}{(n-p)!} c_{i-1,n} (t_i-t_{i-1})^{n-p}  =   p! c_{i,p}, \qquad  0\le p\le m-m_i,    1\le i\le k-1.
\end{equation}

Define the matrix $A$ by
\[
  A  = \begin{pmatrix}
     A_{1}   & B_{1}   & 0         &  \ldots  & 0  & 0    \\
         0     & A_{2} & B_{2} &  \ldots  & 0  & 0    \\
         0     & 0         & A_{3} &  \ldots  & 0  & 0    \\
               & \ldots    &           &  \ldots  &    &      \\
               &           &           &  \ldots  & A_{k-1}  & B_{k-1}
  \end{pmatrix},
\]
where
\[
  A_i  = \begin{pmatrix}
      1    & (t_i-t_{i-1})  & (t_i-t_{i-1})^2  &  \ldots       & (t_i-t_{i-1})^{2m}    \\
      0    & 1              & 2(t_i-t_{i-1})  &  \ldots       & 2m (t_i-t_{i-1})^{2m-1}  \\
      0    & 0              & 2  &  \ldots        & 2m(2m-1) (t_i-t_{i-1})^{2m-2}  \\
           & \ldots   &          &  \ldots        &              \\
      0    &  0       & 0        &  \ldots      & \frac{(2m)!}{(m+m_i)!}(t_i-t_{i-1})^{m+m_i}
  \end{pmatrix}
\]
and
\[
  B_i  = \begin{pmatrix}
      -1    & 0   & 0  &  \ldots  &     & 0    \\
      0    & -1  & 0  &  \ldots  &     & 0   \\
      0    & 0  & -2  &  \ldots  &     & 0   \\
           & \ldots   &          &  \ldots  &     &              \\
      0    &  0       & 0        &  \ldots  & -(m-m_i)! \,\,\, \ldots   & 0
  \end{pmatrix}
\]
are  $(m-m_i+1)\times (2m+1)$  matrices.
Set
\[
  c = (c_{0,0},\ldots, c_{0,2m}, \ldots, c_{k-1,0}, \ldots, c_{k-1,2m} )^t.
\]
We see from (\ref{eqn:e:c}) that
\[
  Ac = 0.
\]
Observe that $\rank(A) = \sum_{i=1}^{k-1} (m-m_i+1)$.
The dimension of the solution set to the above equation is
$(k+1)m + 1+\sum_{i=1}^{k-1} m_i = M_k$.
Hence $\dim |\Vm|^2_{[t_0, t_k]} \le M_k$.
This completes the proof.
\end{proof}

\

\begin{proof}[Proof of Theorem~\ref{thm:main:2}]
For any $f\in \Wm$, there exist $f_1, f_2\in \Vm$ such that
$f(x) =  f_1(x) + \sqrt{-1} f_2(x)$. Hence $|f(x)|^2 = |f_1(x)|^2 + |f_2(x)|^2$.
It follows that $E$ is a linear sampling sequence for $|\Wm|^2_{[t_0, t_k]}$
if and only if it is for $|\Vm|^2_{[t_0, t_k]}$.
In other words, we only need to show that
Theorem~\ref{thm:main:2} is true when $\Wm$ is replaced by $\Vm$.

Let $\tVm$ be the real linear space contains all $2m$-degree B-splines with the same knot sequence $\{t_i\}$ as that of B-splines in $\Vm$
but different knot multiplicity sequence $\{\tilde m_i = m+m_i\}$.

Fix some $k\ge 1$ and $f\in \Vm$. Then $|f|^2$ is $0$ or a $2m$-degree polynomial on every interval $(t_i, t_{i+1})$.
And at every knot $t_i$, $|f|^2$ has up to $(m-m_i)$-th derivative. Hence $|f|^2 \in \tVm$.
Therefore
$|\Vm|^2_{[t_0, t_k]}
\subset \tVm |_{[t_0, t_k]}$.
On the other hand, we see from Lemma~\ref{Lem:spline:basis} and Lemma~\ref{Lem:dimension} that
$\dim |\Vm|^2_{[t_0, t_k]} = M_k =\dim \tVm |_{[t_0, t_k]}$.
Hence
\[
  \Lspan (|\Vm|^2_{[t_0, t_k]}) =  \tVm |_{[t_0, t_k]}.
\]
Now the conclusion follows from Proposition~\ref{prop:local sampling}.
\end{proof}

By setting $t_i=i$ and $m_i=1$, we get Theorem~\ref{thm:linear}.

\subsection{Reconstruction Formula }
Let $E=\{x_n:\, 1\le n\le N\}\subset [t_0, t_k]$ be a sequence which meets (\ref{eqn:ab:1})--(\ref{eqn:ab:4}). We see from Theorem~\ref{thm:main:2}
that there exist functions $S_n\in \Lspan(|\Vm|^2_{[t_0, t_k]})$ such that
\[
  |f(x)|^2 = \sum_{n=1}^N |f(x_n)|^2 S_n(x),\qquad \forall f\in\Wm, \, x\in [t_0, t_n].
\]
In this subsection, we study how to find the functions $S_n$.

Let $\calP_k$ be defined by (\ref{eqn:F}).
Denote functions in $\calP_k$ by $h_i$, $1\le i\le M_k$.
For any $f\in\Vm$, since $\{h_i:\, 1\le i\le M_k\}$ is a basis for
$\Lspan(|\Vm|^2_{[t_0, t_k]})$, there is some $c\in \bbR^{M_k}$ such that
\begin{equation}\label{eqn:fx}
  |f(x)|^2 = \sum_{i=1}^{M_k} c_i h_i(x).
\end{equation}
Hence
\begin{equation}\label{eqn:fxn}
   |f(x_n)|^2 = \sum_{i=1}^{M_k} c_i h_i(x_n),\qquad 1\le n\le N.
\end{equation}
Denote by $\Phi$ the $N\times M_k$ matrix $(h_i(x_n))_{1\le n\le N, 1\le i\le M_k}$.
Set $F = (|f(x_1)|^2$, $\ldots$, $|f(x_N)|^2)^t$. Then  we have
\[
  F = \Phi c.
\]

Since $E$ is a linear sampling sequence for $\tVmk = \Lspan(|\Vm|^2_{[t_0, t_k]})$,
the above equation determines $c$ uniquely. Hence the rank of $\Phi$ is $M_k$.  Consequently,
$c = (\Phi^t\Phi)^{-1} \Phi^t F$.  Set
\[
  (S_1,\ldots, S_n)^t  =  \Phi (\Phi^t\Phi)^{-1} (h_1,\ldots, h_{M_k})^t.
\]
We see from (\ref{eqn:fx}) and (\ref{eqn:fxn}) that
\[
  |f(x)|^2 = \sum_{n=1}^N  |f(x_n)|^2 S_n(x),\qquad x\in [t_0, t_k].
\]


\end{document}